\documentclass[12pt,a4paper]{amsart}
\usepackage{amsmath}
\usepackage{amsthm}
\usepackage{amssymb}
\usepackage{mathrsfs}
\usepackage[english]{babel}
\usepackage{ot1patch}

\newtheorem{thm}{Theorem}[section]

\newtheorem{prop}[thm]{Proposition}

\theoremstyle{remark}
\newtheorem{rem}[thm]{Remark}
\newtheorem*{rem*}{Remark}

\theoremstyle{definition}

\newtheorem{ex}[thm]{Example}

\numberwithin{equation}{section}

\newcommand{\Rz}{\mathbb{R}}

\begin{document}
\title{Erratum to: Medial axis and singularities}

\author{Lev Birbrair \& Maciej P. Denkowski}\address{Universidade Federal do Cear\'a, Fortaleza, Brazil \& Jagiellonian University, Faculty of Mathematics and Computer Science, Institute of Mathematics, \L ojasiewicza 6, 30-348 Krak\'ow, Poland}\email{lev.birbrair@gmail.com \& maciej.denkowski@uj.edu.pl}\date{May 8th 2017}

\begin{abstract}
We correct one erroneous statement made in our recent paper {\it Medial axis and singularities}.
\end{abstract}

\maketitle
\section{Introduction}
In this note `definable' means `definable in an o-minimal structure over the field of real numbers' that in addition is required to be polynomially bounded. 

Let $d(x,X)$ denote the Euclidean distance of $x\in{\Rz}^n$, to $X\subset{\Rz}^n$. We recall that given a closed, nonempty, proper subset $X\subset{\Rz}^n$, we consider its {\it medial axis} as the set defined by
$$
M_X:=\{x\in\Omega\mid \#m(x)>1\}
$$
where $$m(x)=\{y\in X\mid d(x,X)=||x-y||\}$$ is the {\it set of closest points} to $x$ in $X$. A closely related concept is that of the {\it central set} $C_X$ of $X$ that  consists of the centres of maximal balls contained in $\Omega={\Rz}^n\setminus X$ \footnote{An open ball $B\subset\Omega$ is maximal, if for any other ball $B'$ such that $B\subset B'\subset\Omega$, there is $B=B'$.}. It is known that $M_X\subset C_X\subset\overline{M_X}$ (cf. \cite[Theorem 2.25]{BD}).

\medskip
During the preparation of the revised version of our recent paper \cite{BD} we decided, as an afterthought, to include in it the following observation:
\begin{prop}[\cite{BD} Proposition 3.24]
Assume that $X\subset{\Rz}^2$ is a definable curve such that $0\in X$ and the germ $(X\setminus\{0\},0)$ is connected. Then $0\notin \overline{M_X}$.
\end{prop}

Unfortunately, only shortly after the paper had been published we realized that there is a flaw in the proof and this statement is altogether erroneous. This may seem surprising at first sight. Below we give the correct version of the result (Proposition \ref{correction}) preceded by a short introductory preparation. The corrected Proposition has also some mild impact on two other results from \cite{BD} --- see Remark \ref{rk} and Theorem \ref{3.27} hereafter.

\section{Definable plane single branches}

If  $X\subset{\Rz}^2$ is a definable curve such that $0\in X$ and the germ $(X\setminus\{0\},0)$ is connected, i.e. $X$ has a single branch ending at the origin, then the tangent cone $C_0(X)$ is a half-line that we can identify with ${\Rz}_+\times\{0\}^{n-1}\subset{\Rz}^n$ in properly chosen coordinates and $X$ is near zero the graph of a definable $\mathcal{C}^1$ function $f\colon [0,\varepsilon)\to{\Rz}$ with $f(0)=0$ and, clearly, $f'(0)=0$. Then for $0<t\ll 1$, we can write $f(t)=at^\alpha+o(t^\alpha)$ with $a\neq 0, \alpha\geq 1$, provided $f\not\equiv 0$. We say that $X$ is {\it superquadratic} at zero iff $f\not\equiv 0$ and $\alpha<2$ (cf. \cite[Section 3.3]{BD}).

The definability of $f$ allows us also to assume that $f$ has constant convexity on $[0,\varepsilon)$ and is $\mathcal{C}^2$ on $(0,\varepsilon)$.

We shall be using the {\it reaching radius} from \cite[Definition 4.24]{BD} (Section 4 in \cite{BD} is independent from the previous sections).

The correct version of \cite[Proposition 3.24]{BD} reads:
\begin{prop}[\cite{BD} Proposition 3.24 -- correct version]\label{correction}
Assume that $X\subset{\Rz}^2$ is a definable curve such that $0\in X$ and the germ $(X\setminus\{0\},0)$ is connected. Then $0\in \overline{M_X}$ if and only if $X$ is superquadratic at zero.
\end{prop}
\begin{proof}
If $X$ is superquadratic at zero, then by \cite[Lemma 3.17]{BD}, the {\it weak reaching radius} (\cite[Definition 4.24]{BD}) $r'(0,0)$ is zero and so the reaching radius $r(0,0)$ is zero, too. By \cite[Theorem 4.35]{BD}, it means that $0\in \overline{M_X}$. 

If $X$ is not superquadratic at zero, then either $f\equiv 0$, or $\alpha\geq 2$, where $f$ is the function from the argument preceding the Proposition. In both cases $f$ has a $\mathcal{C}^2$ extension by 0 through zero and the Nash Lemma (\cite[Lemma 1.1]{D}) leads to the conclusion that $0\notin\overline{M_X}$.
\end{proof}

\begin{rem}\label{rk}
This result completes \cite[Theorem 3.21]{BD} where now the assumption that the germ $(X\setminus\{0\},0)$ has at least two connected components can be omitted.

Let us also note that that \cite[Theorem 4.35]{BD} together with \cite[Proposition 3.8]{BD} can also be used to simplify the proof of \cite[Theorem 3.19]{BD} (the case when the two branches $\Gamma, \Gamma'$ are superquadratic is a straightforward consequence of the two results cited).
\end{rem}

As we have an additional case in which the medial axis reaches the set $X$, we have to extend the statement of \cite[Theorem 3.27]{BD}. To this aim we prove the following assertion concerning the tangent cone\footnote{Recall that $C_a(E)=\{v\in{\Rz}^n\mid \exists X\ni x_\nu\to a, t_\nu>0\colon t_\nu(x_\nu-a)\to v\}$, for any $E\subset{\Rz}^n$ and $a\in\overline{E}$.}:
\begin{prop}\label{tc}
Assume that $X$ is as in the previous Proposition and $0\in\overline{M_X}\cap X$. Then the tangent cone $C_0(M_X)$ is the half-line perpendicular to $C_0(X)$ lying on the same side of $C_0(X)$ as $X$ near zero. To be more precise, if $X$ near zero is the graph of $f\colon [0,\varepsilon)\to{\Rz}$ and $f$ is, say, convex, then $C_0(M_X)=\{0\}\times[0,+\infty)$.
\end{prop}
\begin{proof}
As in \cite[Theorem 3.27]{BD} we know that $\dim_0 M_X=1$ \footnote{Indeed, $0\in\overline{M_X}\setminus M_X$ and so by  the Curve Selection Lemma, $\dim_0 M_X\geq 1$. On the other hand, $M_X$ has empty interior by the strict convexity of the norm, whence $\dim_0 M_X<2$.}. Again, we assume that $X$ is the graph of a convex definable function $f\colon [0,+\infty)\to{\Rz}$ of class $\mathcal{C}^1$ that is $\mathcal{C}^2$ on $(0,\varepsilon)$, $f(0)=f'(0)=0$ and $f$ is superquadratic at the origin (by Proposition \ref{correction}). Thanks to the convexity, for some neighbourhood $U$ of the origin, we have $M_X\cap U\subset \{(x,y)\in {\Rz}^2\mid y\geq 0\}$.

Take any sequence $M_X\ni a_\nu\to 0$ such that $a_\nu/||a_\nu||\to v$. For each index we pick a point $b_\nu\in m(a_\nu)\setminus\{0\}$. Then $b_\nu\to 0$ (cf. \cite[Lemma 8.5]{Dmfct}). Moreover, $v_\nu:=(a_\nu-b_\nu)/d(a_\nu,X)$ is a unit normal vector to $X$ at $b_\nu$ and for each $\theta\in [0,d(a_\nu,X))$, $b_\nu$ is the unique closest point in $X$ to $b_\nu+\theta v_\nu$ and so the unit vector $v_\nu$ is proximal, which implies, as in the proof of \cite[Theorem 4.35]{BD}, the inequality
$$
\forall c\in X,\> \langle c-b_\nu,v_\nu\rangle\leq \frac{1}{2d(a_\nu,X)}||c-b_\nu||^2.
$$
From this, after multiplying both sides by $d(a_\nu,X)$ and taking $c=0$, we obtain 
$$
\frac{1}{2}||b_\nu||^2\leq \langle a_\nu,b_\nu\rangle,
$$
whence $||b_\nu||/||a_\nu||\leq 2\cos\alpha_\nu$, where $\alpha_\nu=\angle(b_\nu,a_\nu)$. In particular all the angles $\alpha_\nu$ are acute.

Since $||b_\nu||\to 0$, we obtain $b_\nu/||b_\nu||\to (1,0)$, for $C_0(X)=[0,+\infty)\times\{0\}$. Our proof will be accomplished, if we show that $\alpha_\nu\to \pi/2$, since $\alpha_\nu=\angle(b_\nu/||b_\nu||,a_\nu/||a_\nu||)\to \angle ((1,0),v)$. As the angles are acute, we immediately get $\angle((1,0),v)\in [0,\pi/2]$. %Note that the unit normals $v_\nu$ converge to $(0,1)$, because $X$ is $\mathcal{C}^1$. 

We know that $X$ is superquadratic at zero, which implies that for any $y>0$, the origin does not belong to $m((0,y))$ (cf. \cite[Lemma 3.17]{BD}). If $b\in m((0,y))$, then $b$ is the unique closest point for any point from the segment $[(0,y),b]\setminus\{(0,y)\}$. As earlier, by \cite[Lemma 8.5]{Dmfct}, $b\to 0$ when $y\to 0^+$. Then the set 
$
Y:=\{b\in X\mid \exists y>0\colon b\in m((0,y))\}
$
is definable and $0\in\overline{Y}\setminus Y$. Therefore, by the Curve Selection Lemma, $Y$ coincides with $X$ in a neighbourhood of zero that we may take to be a ball $\mathbb{B}(0,R)$. 

In particular, we can find $r,\rho>0$ such that there is a continuous definable surjection $[0,r)\ni y\mapsto F(y)\in X\cap\mathbb{B}(0,\rho)$ satisfying $F(y)\in m((0,y))$. Then, for any $(x,y)\in \mathbb{B}(0,\rho/2)$ such that $x>0, y>f(x)$, the distance $d((x,y),X)$ is realized in $\mathbb{B}(0,\rho)\cap X$. If $b$ is a closest point to $(x,y)$, then the vector $(x,y)-b$ is normal to $X$ at $b$, but as $b=F(y')$ for some $y'\in [0,r)$, we conclude that $(x,y)\in [(0,y'),b]$ and so $m((x,y))=\{b\}$. Therefore, $$M_X\cap \{(x,y)\in\mathbb{B}(0,\rho/2)\mid x>0,y>0\}=\varnothing.$$

This means that $M_X\cap \mathbb{B}(0,\rho/2)\subset\{(x,y)\in{\Rz}^2\mid y\geq 0, x\leq 0\}$, whence $\angle((1,0),v)\in [\pi/2,\pi]$. Summing up, we obtain $\angle((1,0),v)=\pi/2$ as required.
\end{proof}

The correct version of Proposition \ref{correction} has also some impact on \cite[Theorem 3.27]{BD} in that we have to slightly modify its statement and include in its proof one more case. Before we state it, we need to recall a few things from \cite{BD}. 

If $(X,0)\subset{\Rz}^2$ is a definable pure one-dimensional closed germ, then $X\setminus\{0\}$ consist of finitely many branches $\Gamma_0,\dots, \Gamma_{k-1}$ ending at zero and dividing a small ball $\mathbb{B}(0,r)$ into $k$ regions. For $k>1$, if we enumerate the branches in a consecutive way, we can call these open regions $D(\Gamma_i,\Gamma_{i+1})$, $i\in \mathbb{Z}_k$. Assuming that $0\in \overline{M_X}$, we say that a pair of {\it consecutive} branches $\Gamma_i,\Gamma_{i+1}$ {\it contributes} to $M_X$ at zero, if $0\in\overline{M_X\cap D(\Gamma_i,\Gamma_{i+1})}$. 

Let $1\leq c\leq k$ be the number of contributing regions. For each such region $D(\Gamma_i,\Gamma_{i+1})$ we have two half-lines $\ell_i,\ell_{i+1}$ tangent to $\Gamma_i,\Gamma_{i+1}$ at zero, respectively. These half-lines define an oriented angle $\alpha(i,i+1)\in [0,2\pi]$, consistent with the region\footnote{Note that it may happen that $\alpha(i,i+1)=2\pi$; indeed, if $X$ consists of the two branches $\Gamma_0=[0,+\infty)\times\{0\}$ and the superquadratic $\Gamma_1=\{y=x^{3/2}, x\geq 0\}$, then both regions $D(\Gamma_0,\Gamma_1)$ and $D(\Gamma_1,\Gamma_0)$ are contributing. The angles are $0$ and $2\pi$, respectively.}. 

As we know that $M_X$ is one-dimensional, the germ $(\overline{M_X},0)$ consists of finitely many branches ending at zero. For a definable curve germ $(E,0)$, we will denote by $b_0(E)$ the number of its branches at the origin.
\begin{thm}[\cite{BD} Theorem 3.27]\label{3.27}
Assume that $0\in\overline{M_X}\cap X$ where $X$ is a pure one-dimensional closed definable set in the plane. Then,\begin{enumerate}
\item either $b_0(X)=1$, in which case $b_0(M_X)=1$ and $C_0(M_X)$ is the half-line perpendicular to $C_0(X)$ lying on the same side of $C_0(X)$ as $X$ near zero,
\item or $b_0(X)=k>1$, in which case $b_0(M_X)\leq c+1$ where $c$ is the number of contributing regions, and $C_0(M_X)$ is the union of the bisectors of all the pairs of half-lines forming up $C_0(X)$ given by pairs of consecutive branches delimiting regions that contribute to $M_X$ at zero with possibly one exception:\\ there is at most one contributing region $D(\Gamma_i,\Gamma_{i+1})$ with angle $\alpha(i,i+1)>\pi$ in which case at least one of the curves $\Gamma_i,\Gamma_{i+1}$ is superquadratic at zero and $M_{i,i+1}=M_X\cap D(\Gamma_i,\Gamma_{i+1})$ has at most two branches at zero and $C_0(M_{i,i+1})$ consists of one or two half-lines orthogonal to the corresponding tangent $\ell_i$ or $\ell_{i+1}$.
\end{enumerate}
\end{thm}

\begin{proof}
(1) is the statement of Proposition \ref{tc}. To see that $M_X$ near zero consists of one branch we consider the situation from the proof of Proposition \ref{tc}. In particular $M_X\cap\mathbb{B}(0,r)\subset\{(x,y)\in{\Rz}^2\mid y\geq 0, x\leq 0\}$. Suppose that there are (at least) two different branches $M_1,M_2$ ending at zero. Then one of them, say $M_1$, lies in the region delimited by the other one, i.e. $M_2$, and $\{0\}\times [0,+\infty)$. Take a point $a\in M_2$. Then $m(a)$ contains a non-zero point $b$. Then, if $a$ is sufficiently near zero, the segment $[a,b]$ intersects $M_1$. If $c$ belongs to the intersection, then $m(c)=\{b\}$, contrary to $c\in M_X$.

% and let $a\notin M_X$ be a point lying between them. Then it $m(a)=\{b\}$ and if $a$ is sufficiently near zero, the segment $[a,b]$ must intersect $M_1\cup M_2$, provided $b\neq 0$. Let $c$ be the point of intersection in this case. Then $b\in m(c)$ and there is another point $b'\in m(c)\setminus\{b\}$. But $||c-b||=||c-b'||$ and $||a-b'||<||a-b||$ contrary to $b\in m(a)$. 
%
%On the other hand, if $b=0$, then either $[a,0]\cap (M_1\cup M_2)\neq\varnothing$ in which case we may repeat the preceding argument, or $[a,0]$ is disjoint with $M_1\cup M_2$. If the latter occurs, $M_1,M_2$ lie on different sides of the segment $[a,0]$ near zero.   

As for (2), we can repeat the argument from the proof in \cite{BD} with only one additional case to consider. Let $D(\Gamma_0,\Gamma_1)$ be a contributing region. The same type of argument as above shows that $M_X$ has only one branch in $D(\Gamma_0,\Gamma_1)$ ending at zero\footnote{If there were only two branches of $M_X$ in $D(\Gamma_0,\Gamma_1)$ ending at zero, it could happen that along each of them the segments joining the points to the points realizing their distance would not intersect the other branch. In that case we pick a point $a$ in between the two branches of $M_X$ and the segment $[a,m(a)]$ must intersect one of the branches in a point $c$. Then $m(a)\in m(c)$ but there is a point $b\in m(c)\setminus m(a)$ and the triangle inequality shows that $||a-b||<||a-m(a)||$, which is a contradiction.}. Let  $\alpha=\alpha(0,1)\in[0,2\pi]$ be the oriented angle consistent with $D(\Gamma_0,\Gamma_1)$. 

If $\alpha \in [0,\pi)$, we proceed as in \cite[Theorem 3.27]{BD}: for $a\in M_X$ near zero, $m(a)$ cannot contain zero and has points both from $\Gamma_0$ and $\Gamma_{1}$ --- these tend to zero when $a\to 0$.  The set $M_X\cap D(\Gamma_0,\Gamma_{1})$ coincides with the conflict set of $\Gamma_0,\Gamma_{1}$ (compare the proof of Theorem 3.21 in \cite{BD}) and the Birbrair-Siersma Theorem from \cite{BS} (see also \cite[Theorem 3.26]{BD}) gives the result as in the original proof in \cite{BD}.

If $\alpha=\pi$, then $\Gamma=\Gamma_0\cup\Gamma_1$ is a $\mathcal{C}^1$ curve and $M_X\cap D(\Gamma_0,\Gamma_1)$ reaches the origin iff $\Gamma$ is superquadratic at zero \footnote{I.e. $D(\Gamma_0,\Gamma_1)$ is near zero the epigraph of a superquadratic function.}. But then no point from the normal cone at zero can have its distance realized at the origin (cf. \cite[Lemma 3.17]{BD}) and so we are in a position that allows us to repeat the argument based on the Birbrair-Siersma Theorem just as in \cite{BD}.

If $\alpha>\pi$ (clearly, there can be only one such contributing region), then the only possibility that the region $D(\Gamma_0,\Gamma_1)$ be contributing is that at least one of the two delimiting curves be superquadratic at zero and $\mathbb{B}(0,r)\setminus D(\Gamma_0,\Gamma_1)$ be non-convex. In this case we are exactly in the situation from Proposition \ref{tc} and the result follows. Of course, $M_X\cap D(\Gamma_0,\Gamma_1)$ may have two branches at zero which explains why we have $b_0(M_X)\leq c+1$.
\end{proof}
\begin{ex}
Rotate the superquadratic curve $y=x^{3/2}$, $x\geq 0$ by $\pi/6$ clockwise and the curve $y=-x^{3/2}$, $x\geq 0$ by the same angle anticlockwise, obtaining two curves $\Gamma_0,\Gamma_1$ with tangent half-lines at zero $y=(1/\sqrt{3})x, x\geq 0$ and $y=-(1/\sqrt{3})x, x\geq 0$, respectively. Let $X=\Gamma_0\cup\Gamma_1$. Then we have two contributing regions: $D(\Gamma_1,\Gamma_0)$ with $\alpha(1,0)=\pi/3$ and $D(\Gamma_0,\Gamma_1)$ with $\alpha(1,2)=5\pi/3$. The medial axes has three branches at zero: the half-line $[0,+\infty)\times\{0\}$ and two curves symmetric with respect to $(-\infty,0]\times\{0\}$, living in the quadrants $\{x\leq 0, y\geq 0\}$ and $\{x\leq 0, y\geq 0\}$, respectively. Then 
$$
C_0(M_X)=([0,+\infty)\times\{0\})\cup\{y=-\sqrt{3}x, x\leq 0\}\cup\{y=\sqrt{3}x, x\leq 0\}.
$$
\end{ex}
\begin{rem}
In our article \cite{BD} there is one misprint in Proposition 3.8 that definitely should be corrected as it makes the statement unclear. Namely the set $S$ from Proposition 3.8 should be defined as $S=B\cap\mathbb{S}(m(x_0),d(x_0))$, i.e. the sphere is centred at $m(x_0)$ (not at $x_0$ as appeared in print).
\end{rem}

\end{document}